\documentclass{amsart}
\usepackage{amssymb}
\usepackage[initials]{amsrefs}
\usepackage{tikz-cd}
\usepackage[english]{babel}

\usepackage{graphicx}
\usepackage{latexsym}
\usepackage{amsmath,amssymb,amsthm,amsfonts,amscd}
\usepackage{amssymb, mathrsfs, enumerate}
\usepackage{mathtools}
\usepackage{colonequals}
\usepackage[dvipsnames]{xcolor}
\usepackage{comment}

\renewcommand{\epsilon}{\varepsilon}

\usepackage{mleftright}

\usepackage{tikz}	
\usepackage{pgfplots}	
\usetikzlibrary{arrows, angles, quotes, calc,through,backgrounds,matrix,decorations.markings,decorations.pathmorphing,
	intersections, pgfplots.fillbetween, patterns}
 \usepgfplotslibrary{polar}
\pgfplotsset{compat=newest}
\usepackage{multirow}
\usepgfplotslibrary{fillbetween}
\pgfplotsset{compat=1.17}
\pgfdeclarelayer{ft}
\pgfdeclarelayer{bg}
\pgfsetlayers{bg,main,ft}

\usepackage{thmtools}

\usepackage[linkcolor=black, urlcolor=black, citecolor=black, colorlinks=true, hypertexnames=false]{hyperref}
\usepackage[capitalize, nameinlink, noabbrev]{cleveref} 

\newcommand{\progr}{\mathcal{A}}
\newcommand{\graph}{\mathcal{G}}

\newcommand{\tree}{\mathcal{T}}

\newcommand{\aut}{\operatorname{Aut}}

\newcommand{\Sym}{\operatorname{Sym}}
\newcommand{\universal}{\mathcal{U}}
\newcommand{\id}{\operatorname{id}}

\renewcommand*{\S}{\mathcal S}

\newtheorem{thm}{Theorem}[section]

\newtheorem{prop}[thm]{Proposition}

\theoremstyle{definition}

\newtheorem{defn}[thm]{Definition}

\newtheorem{construction}[thm]{Construction}
\newtheorem{ex}[thm]{Example}

\newtheorem{rem}[thm]{Remark}

\title[Quasi-isometries between graphs of tdlc groups]{Quasi-isometries between graphs of totally disconnected locally compact groups}

\author{Sebastian Giersbach}
\address{Mathematisches Institut, Justus-Liebig-Universität Gießen, 35392 Gießen, Germany}
\email{sebastian.giersbach@uni-giessen.de}

\subjclass{Primary 20F65; Secondary 20E06, 20E08, 22D05}

\date{}

\begin{document}

\begin{abstract}
    Let $G$ and $H$ be compactly generated totally disconnected locally compact (tdlc) groups that decompose as finite graphs of tdlc groups $(\graph, \progr)$ and $(\mathcal{H}, \mathcal{B})$ such that all edge groups are compact and all vertex groups have at most one end. We generalize a result of Papasoglu--Whyte to tdlc groups and show that $G$ and $H$ are quasi-isometric if and only if they have the same number of ends and every one-ended vertex group of $(\graph, \progr)$ is quasi-isometric to a one-ended vertex group of $(\mathcal{H}, \mathcal{B})$, and vice versa. As an application, we construct uncountably many pairwise non-quasi-isometric compactly generated non-discrete simple tdlc groups, strengthening a result by Smith.
\end{abstract}

\maketitle

\section{Introduction}

Stallings' ends theorem \cite{stallings}*{5.A.10} states that a finitely generated group has more than one end if and only if it splits as an amalgam or an HNN extension over a finite subgroup. This motivates the notion of a \emph{terminal} graph of groups: It is a finite graph of groups consisting of finite edge groups and finitely generated vertex groups with at most one end. A finitely generated group that decomposes as such is called \emph{accessible}. Papasoglu and Whyte \cite{papasoglu+whyte}*{Theorem 0.4} showed that the quasi-isometry class of an accessible group depends only on its number of ends and the set of quasi-isometry classes of one-ended vertex groups in a terminal decomposition. By Dunwoody \cite{dunwoody1985}*{Theorem 5.1}, every finitely presented group is accessible. In particular, the quasi-isometry classification of finitely presented groups reduces to that of one-ended groups.

In this paper, we are interested in totally disconnected locally compact (tdlc) groups. By Abels \cite{abels}*{Struktursatz 5.7}, the analogue of Stallings' ends theorem also holds for tdlc groups: A compactly generated tdlc group has more than one end if and only if it splits as an amalgam or an HNN extension over a compact open subgroup. Thus, the definition of an accessible group generalizes to tdlc groups. We extend Papasoglu and Whyte's theorem to tdlc groups.

\begin{thm}\label{thm: Papasoglu-Whyte for tdlc groups}
    Let $G$ and $H$ be accessible compactly generated tdlc groups with decompositions $(\graph, \progr)$ and $(\mathcal{H}, \mathcal{B})$ into terminal graphs of tdlc groups. The groups $G$ and $H$ are quasi-isometric if and only if $G$ and $H$ have the same number of ends and $(\graph, \progr)$ and $(\mathcal{H}, \mathcal{B})$ have the same set of quasi-isometry classes of one-ended vertex groups.
\end{thm}

In order to show the discrete version, Papasoglu and Whyte introduced free products of graphs; we state their definition and basic properties in \cref{sec: free products of graphs}. The implication that quasi-isometric accessible groups have the same set of quasi-isometry classes of one-ended vertex groups, generalizes to tdlc groups with little preparation needed. The other implication holds without assumption on the vertex groups, see \cref{thm: quasi-isometric factors implies quasi-isometric graph of groups}. When proving this implication we treat amalgams and HNN extensions separately. While only little work has to be done for amalgams, HNN extensions require a separate proof. The proof of \cref{thm: Papasoglu-Whyte for tdlc groups} is the content of \cref{sec: generalization to tdlc groups}.

In \cref{sec: non-quasi-isometric simple tdlc groups}, we strengthen a result by Smith \cite{smith}*{Corollary 39}. Denote by $\S$ the class of compactly generated non-discrete tdlc groups that are topologically simple. By introducing a group $\universal(M, N)$ associated to a pair of permutation groups $M$ and $N$, Smith constructed $2^{\aleph_0}$ pairwise non-isomorphic abstractly simple groups in $\S$. As the Smith group $\universal(M, N)$ splits under mild conditions as an amalgam over a compact open subgroup with factors quasi-isometric to $M$ and $N$, \cref{thm: Papasoglu-Whyte for tdlc groups} is applicable. By choosing appropriate groups $M$ and $N$, we get the following result.

\begin{thm}\label{thm: Uncountably many quasi-isometry classes in S}
    There are precisely $2^{\aleph_0}$ quasi-isometry classes of abstractly simple groups in $\S$.
\end{thm}

The discrete statement, i.e. that there exist $2^{\aleph_0}$ quasi-isometry classes of finitely generated simple groups, was proven by Minasyan, Osin and Witzel \cite{minasyan+osin+witzel}. Furthermore, we can arrange our examples to not contain any lattices.

\medskip

After the first version of this paper appeared on the arXiv, Ravi Tomar informed the author that he and Arunava Mandal independently obtained \cref{thm: Papasoglu-Whyte for tdlc groups}. Their result has since been published \cite{mandal+tomar}.

\section{Preliminaries}\label{sec: preliminaries}

\subsection{Amalgams and HNN extensions of tdlc groups}

A group that decomposes as a finite graph of groups can be described by iteratively taking amalgams and HNN extensions. Thus, we will reduce some theorems to these base cases. For properties of amalgams and HNN extensions as abstract groups, we refer the reader to Lyndon and Schupp \cite{lyndon+schupp}*{Chapter IV}. As we are interested in tdlc groups, we also require their topological versions. Therefore, we briefly recall the definition and basic properties of amalgams and HNN extensions of topological groups, following Cornulier and de la Harpe \cite{cornulier+delaharpe}*{Chapter 8.B}.

\begin{prop}[\cite{cornulier+delaharpe}*{Propositions 8.B.9 and 8.B.10}]\label{prop: amalgams and HNN extensions of topological groups}
    \begin{enumerate}
        \item Let $A$ and $B$ be topological groups and let $C$ be a common open subgroup of $A$ and $B$. There exists a unique group topology on the amalgam $G = A *_C B$ such that $C$ is an open subgroup of $G$. Equipped with this topology, we call $G$ an \emph{amalgam of topological groups}.

        \item Let $A$ be a topological group and let $C$ and $C'$ be two open subgroups of $A$ with an isomorphism $\varphi \colon C \to C'$. There exists a unique group topology on the HNN extension $G = A *_\varphi$ such that $A$ is an open subgroup of $G$. Equipped with this topology, we call $G$ an \emph{HNN extension of topological groups}. We also write $A *_C$ for $A *_\varphi$.
    \end{enumerate}
\end{prop}

Suppose $G$ is an amalgam or HNN extension of topological groups. If the factors are tdlc groups, then so is $G$. By iteratively taking amalgams and HNN extensions of tdlc groups, we can define graphs of tdlc groups and equip their fundamental groups with a tdlc group topology.

\begin{defn}
    Let $(\graph, \progr)$ be a graph of groups. We call it a \emph{graph of tdlc groups} if all vertex groups $\progr_v$ for $v \in V(\graph)$ and all edge groups $\progr_e$ for $e \in E(\graph)$ are tdlc groups and if all edge homomorphisms $\progr_e \hookrightarrow \progr_v$ are continuous open embeddings.

    Suppose $(\graph, \progr)$ is a finite graph of tdlc groups with abstract fundamental group $G$. As $G$ can be described by iteratively taking amalgams and HNN extensions, we equip $G$ with the unique tdlc group topology such that every vertex group $\progr_v$ and every edge group $\progr_e$ is open in $G$ by \cref{prop: amalgams and HNN extensions of topological groups}.
\end{defn}

\subsection{Cayley--Abels graphs}

When talking about a finitely generated group as a geometric object, one usually studies its Cayley graph over some finite generating set. The analogue of Cayley graphs for compactly generated tdlc groups is the Cayley--Abels graph. For a thorough introduction on Cayley--Abels graphs, we refer the reader to Kr\"{o}n and M\"{o}ller \cite{kroen+moeller}.

\begin{defn}
    Let $G$ be a locally compact group. A \emph{Cayley--Abels graph} of $G$ is a locally finite connected graph on which $G$ acts vertex-transitively with compact open vertex stabilizers.
\end{defn}

A locally compact group admits a Cayley--Abels graph if and only if it is compactly generated and has a compact open subgroup, see \cite{cornulier+delaharpe}*{Proposition 2.E.9}. By van Dantzig's theorem, any tdlc group has a compact open subgroup. Thus, a tdlc group admits a Cayley--Abels graph if and only if it is compactly generated. One possible way to construct a Cayley--Abels graph is the following.

\begin{construction}[\cite{kroen+moeller}*{Construction 2}]\label{constr: definition of a Cayley--Abels graph}
    Let $G$ be a compactly generated tdlc group and let $U$ be a compact open subgroup of $G$. The left $U$-cosets form an open cover of any compact generating set of $G$. Thus, there are finitely many elements $g_1, \ldots, g_n \in G$ such that $G = \left< U, g_1, \ldots, g_n \right>$. We assume this generating set to be symmetric. Let $\Gamma$ be the Cayley graph of $G$ for this generating set. Consider the quotient graph of $\Gamma$ by identifying elements of the same left $U$-coset. Its vertex set is $G/C$ and two distinct cosets $g_0C$ and $h_0C$ are adjacent if there are elements $g \in g_0C$ and $h \in h_0C$ with $g = hg_i$ for some $1 \le i \le n$. This yields a Cayley--Abels graph of $G$. For example, the vertex stabilizers are the conjugates of $C$ and thus compact open subgroups. If the additional generators $g_1, \ldots, g_n$ are clear or do not matter, we denote the Cayley--Abels graph of $G$ with vertex set $G/C$ by $\Gamma_{G,C}$.
\end{construction}

As for Cayley graphs of finitely generated groups, any two Cayley--Abels graphs of a compactly generated tdlc group are quasi-isometric. Hence, the number of ends of a tdlc group is the number of ends of any of its Cayley--Abels graphs. As Cayley--Abels graphs are locally finite vertex-transitive graphs, a compactly generated tdlc group has either no ends, one end, two ends or infinitely many ends. The tdlc groups with no ends are the compact groups. Stallings' ends theorem generalizes to tdlc groups, characterizing tdlc groups with more than one end. Abels \cite{abels}*{Struktursatz 5.7} proved it for compactly generated locally compact groups in terms of Specker compactifications. Kr\"{o}n and M\"{o}ller \cite{kroen+moeller}*{Theorem 3.18} proved the same result for compactly generated tdlc groups using Bass--Serre theory.

\begin{thm}[Stallings' ends theorem for tdlc groups]
    Let $G$ be a compactly generated tdlc group. Then $G$ has at least two ends if and only if $G$ splits non-trivially over a compact open subgroup, i.e. if one of the following two cases holds:
    \begin{itemize}
        \item $G = A *_C B$ is an amalgam of tdlc groups, where $A$ and $B$ are compactly generated open subgroups of $G$, and $C$ is a proper compact open subgroup of both $A$ and $B$;

        \item $G = A *_C$ is an HNN extension of tdlc groups, where $A$ is a compactly generated open subgroup of $G$, and $C$ is a compact open subgroup of $A$.
    \end{itemize}
\end{thm}

Suppose a finitely generated group $G$ has at least two ends. By Stallings' ends theorem, we have $G = A *_C B$ or $G = A *_C$ where $A$ and $B$ are finitely generated and $C$ is a finite subgroup. If some vertex group has at least two ends, we iterate by splitting it further. This process might terminate, in which case we get a decomposition of $G$ into a finite graph of groups $(\graph, \progr)$ with finite edge groups and vertex groups with at most one end. By Dunwoody \cite{dunwoody1985}*{Theorem 5.1}, every finitely presented group is accessible. However, there exist inaccessible groups, as first shown by Dunwoody \cite{dunwoody1993}. In a similar way, we define accessible tdlc groups.

\begin{defn}
    Let $G$ be a compactly generated tdlc group. It is \emph{accessible} if $G$ decomposes as a finite graph of tdlc groups $(\graph, \progr)$ such that the edge groups are compact open subgroups of $G$ and the vertex groups are compactly generated open subgroups of $G$ with at most one end. We call $(\graph, \progr)$ a \emph{terminal} graph of tdlc groups.
\end{defn}

Dunwoody's theorem can be generalized to tdlc groups. Every compactly presented tdlc group is accessible, see \cite{cornulier}*{Theorem 19.45}.

\section{Free products of graphs}\label{sec: free products of graphs}

As groups decomposing as finite graphs of groups are iterated amalgams and HNN extensions, it suffices to consider the base case of a single amalgam or a single HNN extension. One implication of Papasoglu and Whyte's main theorem essentially reduces to the following statement on free products.

\begin{thm}[\cite{papasoglu+whyte}*{Theorem 0.1}]\label{thm: quasi-isometric factors implies quasi-isometric free products of groups}
    Let $A, B$ and $C$ be non-trivial groups. If $A$ and $B$ are quasi-isometric, then the free products $A * C$ and $B * C$ are quasi-isometric unless $C$, and one of $A$ or $B$, has order $2$.     
\end{thm}

Papasoglu and Whyte prove this theorem by introducing a free product of graphs, proving the analogous statement and then applying it to Cayley graphs.

\begin{construction}
    Let $(X, x_0)$ and $(Y, y_0)$ be two graphs with base points $x_0$ and $y_0$. The \emph{free product} $(X, x_0) * (Y, y_0)$ of these graphs will consist of copies of $X$ and $Y$ with edges connecting these copies. We construct it recursively as follows.

    Let $\Gamma_0$ be the disjoint union of $X$ and $Y$ with an edge between $x_0$ and $y_0$. We call this edge the \emph{base edge} of $(X, x_0) * (Y, y_0)$. Given $\Gamma_n$, we construct $\Gamma_{n+1}$: For any vertex $v \in \Gamma_n$ not adjacent to another subgraph, add a subgraph isomorphic to $X$ or $Y$ and add an edge between $v$ and the base point of the new subgraph. For instance, if $v$ is a vertex of a copy of $X$, add a subgraph isomorphic to $Y$ and add an edge connecting $v$ to the base point $y_0$ of this new copy of $Y$. If $v$ is a vertex of a copy of $Y$, add a subgraph isomorphic to $X$ instead. We have $\Gamma_n$ as a subgraph of $\Gamma_{n+1}$. Now we define $(X, x_0) * (Y, y_0)$ to be the direct limit of the graphs $\Gamma_n$, i.e. $(X, x_0) * (Y, y_0) = \bigcup_n \Gamma_n$.
\end{construction}

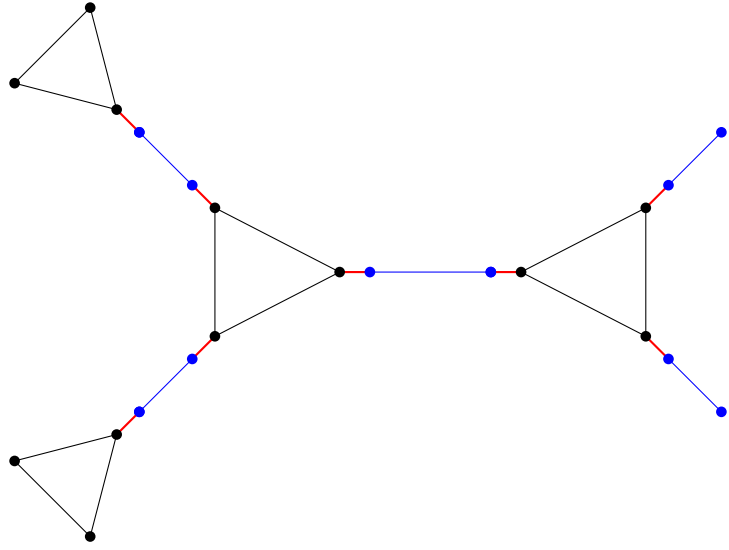
\begin{figure}[htbp]
    
    \centering
    
        \begin{tikzpicture}

        \draw (-0.2,0)--(-1.85,0.85)--(-1.85,-0.85)--(-0.2,0);
        \draw[red,thick] (-0.2,0)--(0.2,0);
        \draw[red,thick] (-1.85,0.85)--(-2.15,1.15);
        \draw[red,thick] (-1.85,-0.85)--(-2.15,-1.15);
        
        \fill[black] (-0.2,0) circle (2pt);
        \fill[black] (-1.85,0.85) circle (2pt);
        \fill[black] (-1.85,-0.85) circle (2pt);
        
        \draw[blue] (0.2,0)--(1.8,0);
        \fill[blue] (0.2,0) circle(2pt);
        \fill[blue] (1.8,0) circle(2pt);

        \draw[blue] (-2.15,1.15)--(-2.85,1.85);
        \fill[blue] (-2.15,1.15) circle (2pt);
        \fill[blue] (-2.85,1.85) circle (2pt);
        
        \draw[blue] (-2.15,-1.15)--(-2.85,-1.85);
        \fill[blue] (-2.15,-1.15) circle (2pt);
        \fill[blue] (-2.85,-1.85) circle (2pt);

        \draw (2.2,0)--(3.85,0.85)--(3.85,-0.85)--(2.2,0);
        \draw[red,thick] (1.8,0)--(2.2,0);
        \draw[red,thick] (3.85,0.85)--(4.15,1.15);
        \draw[red,thick] (3.85,-0.85)--(4.15,-1.15);
        \fill[blue] (1.8,0) circle(2pt);

        \fill[black] (2.2,0) circle (2pt);
        \fill[black] (3.85,0.85) circle (2pt);
        \fill[black] (3.85,-0.85) circle (2pt);

        \draw[blue] (4.15,1.15)--(4.85,1.85);
        \fill[blue] (4.15,1.15) circle (2pt);
        \fill[blue] (4.85,1.85) circle (2pt);
        
        \draw[blue] (4.15,-1.15)--(4.85,-1.85);
        \fill[blue] (4.15,-1.15) circle (2pt);
        \fill[blue] (4.85,-1.85) circle (2pt);

        \draw (-3.15,2.15)--(-4.5,2.5)--(-3.5,3.5)--(-3.15,2.15);
        \draw[red,thick] (-2.85,1.85)--(-3.15,2.15);
        \fill[blue] (-2.85,1.85) circle (2pt);

        \fill[black] (-3.15,2.15) circle (2pt);
        \fill[black] (-4.5,2.5) circle (2pt);
        \fill[black] (-3.5,3.5) circle (2pt);

        \draw (-3.15,-2.15)--(-4.5,-2.5)--(-3.5,-3.5)--(-3.15,-2.15);
        \draw[red,thick] (-2.85,-1.85)--(-3.15,-2.15);
        \fill[blue] (-2.85,-1.85) circle (2pt);
        
        \fill[black] (-3.15,-2.15) circle (2pt);
        \fill[black] (-4.5,-2.5) circle (2pt);
        \fill[black] (-3.5,-3.5) circle (2pt);
        
        \end{tikzpicture}

    \caption{The subgraph $\Gamma_2$ of the free product of the cycle graphs $C_3$ (drawn in black) and $C_2$ (drawn in blue). The edges connecting different subgraphs are drawn in red.}
        
\end{figure}

If $X$ and $Y$ are vertex-transitive, e.g. Cayley or Cayley--Abels graphs, then all choices of base points $x_0$ and $y_0$ lead to isomorphic free products $(X, x_0) * (Y, y_0)$. Even if $X$ and $Y$ are not vertex-transitive, for a large class of graphs the bilipschitz class of $(X, x_0) * (Y, y_0)$ does not depend on the choice of base points. Hence, we will simply write $X * Y$.

\begin{defn}
    A graph $\Gamma$ is \emph{homogeneous} if for some $L \ge 1$ and for any two vertices $v_1, v_2 \in V(\Gamma)$ there is an $L$-bilipschitz map $f \colon \Gamma \to \Gamma$ with $f(v_1) = v_2$.
\end{defn}

\begin{prop}[\cite{papasoglu+whyte}*{Lemma 1.1}]\label{prop: bilipschitz class of free products of graphs}
    Let $X$ and $Y$ be homogeneous graphs. Let $\Gamma$ be a graph such that:
    \begin{itemize}
        \item $\Gamma$ contains a disjoint family of subgraphs $(X_i)_i$ and $(Y_i)_i$ whose union contains all vertices of $\Gamma$.

        \item Every edge of $\Gamma$ not in one of the subgraphs connects an $X_i$ to a $Y_j$. There is exactly one such edge at every vertex of $\Gamma$.

        \item There is an $L \ge 1$ such that for every $i$ there are $L$-bilipschitz equivalences $X_i \to X$ and $Y_i \to Y$.

        \item The quotient graph obtained from $\Gamma$ by collapsing each of the subgraphs to a point is a tree.
    \end{itemize}

    Then for any edge $e$ in $\Gamma$ connecting an $X_i$ to a $Y_j$ and any choice of base vertices $x_0$ of $X$ and $y_0$ of $Y$ there is an $M$-bilipschitz equivalence $\Gamma \to (X, x_0) * (Y, y_0)$ mapping $e$ to the base edge. The bilipschitz constant $M$ depends only on $L$ and the homogeneity constants of $X$ and $Y$.
\end{prop}

In particular, this proposition shows that the free product of homogeneous graphs is itself homogeneous.

\begin{rem}
    Let $X * Y$ be a free product of two homogeneous graphs. Contracting all edges connecting different subgraphs yields a quasi-isometry to a graph $X + Y$. This graph is a special case of tree amalgamations. For a thorough introduction to tree amalgamations see Hamann \cite{hamann} or Mohar \cite{mohar}. Note that our notation is opposite to Hamann's. Some results for graphs of groups carry over. Hamann, Lehner, Miraftab and Rühmann \cite{hamann+lehner+miraftab+ruehmann}*{Theorem 1.1} have proven a version of Stallings' ends theorem for tree amalgamations. Hamann \cite{hamann}*{Theorem 1.3} has shown that an analogue of Papasoglu and Whyte's theorem holds for tree amalgamations.

    If $\Gamma_A$ is the Cayley graph of a group $A$ over a generating set $S$ and $\Gamma_B$ is the Cayley graph of a group $B$ over a generating set $T$, then the Cayley graph of $A * B$ over the generating set $S \cup T$ is in general not isomorphic to the free product of graphs $\Gamma_A * \Gamma_B$ but instead to the tree amalgamation $\Gamma_A + \Gamma_B$. As we're only interested in graphs up to quasi-isometry, we won't distinguish between them and say nonetheless that the free product of graphs $\Gamma_A * \Gamma_B$ is a Cayley graph of $A * B$.

    Both operations $X * Y$ and $X + Y$ on graphs are commutative. Furthermore, the operation $X + Y$ is associative. From \cref{thm: quasi-isometric factors implies quasi-isometric free products of graphs} it follows that the free products $X * (Y * Z)$ and $(X * Y) * Z$ of homogeneous graphs are quasi-isometric as long as all factors have at least two vertices.
\end{rem}

Following Papasoglu and Whyte's proof of \cref{thm: quasi-isometric factors implies quasi-isometric free products of groups} for infinite factors, they actually show a more general statement. 

\begin{thm}[\cite{papasoglu+whyte}*{Theorem 0.1}]\label{thm: quasi-isometric infinite factors implies quasi-isometric free products of graphs}
    Let $X, Y$ and $Z$ be infinite homogeneous graphs. If $X$ and $Y$ are quasi-isometric, then $X * Z$ and $Y * Z$ are quasi-isometric.
\end{thm}

We will not reproduce the proof of their theorem here. However, we will need a lemma used in its proof, so we state it here for the reader's convenience.

\begin{prop}[\cite{papasoglu+whyte}*{Lemma 1.4}]\label{prop: X * Y is bilipschitz equivalent to X * (Y * Y)}
    Let $X$ and $Y$ be infinite homogeneous graphs. Then the free products $X * Y$ and $X * (Y * Y)$ are bilipschitz equivalent.
\end{prop}

Papasoglu and Whyte remark that \cref{thm: quasi-isometric infinite factors implies quasi-isometric free products of graphs} can easily be extended to certain finite factors using similar arguments to their proof of \cref{thm: quasi-isometric factors implies quasi-isometric free products of groups}, see \cite{papasoglu+whyte}*{Theorem 0.1}, but leave the precise statement as an exercise for the reader. As we will need the theorem for finite factors, we give a full proof.

\begin{thm}\label{thm: quasi-isometric factors implies quasi-isometric free products of graphs}
    Let $X, Y$ and $Z$ be homogeneous graphs. If $X$ and $Y$ are quasi-isometric, then $X * Z$ and $Y * Z$ are quasi-isometric, except in the following cases:
    \begin{itemize}
        \item $X$ or $Y$ has only a single vertex.

        \item $Z$ and one of $X$ and $Y$ have exactly two vertices.
    \end{itemize}
\end{thm}

\begin{proof}
    Note that any finite graph is homogeneous, as any permutation of the vertices is a bilipschitz equivalence. If all graphs are infinite, then the statement follows immediately from \cref{thm: quasi-isometric infinite factors implies quasi-isometric free products of graphs}.

    First, suppose that $X$ (hence also $Y$) and $Z$ are finite with $|X| = \ell$, $|Y| = m$ and $|Z| = n$. In the free product $X * Z$ project every vertex in each copy of $X$ and $Z$ onto their respective base vertex. This yields a quasi-isometry from $X * Z$ to the $(\ell, n)$-biregular tree $T_{\ell, n}$. Similarly, $Y * Z$ is quasi-isometric to $T_{m, n}$. The statement follows from the fact that for finite $p$ and $q$ all biregular trees $T_{p, q}$ are quasi-isometric to the $3$-regular tree $T_3$ with the exception of $T_{1, q}$ and $T_{p, 1}$, which are finite, and $T_{2, 2}$, which is isometric to $T_2$.

    Next, suppose that $X$ and $Y$ are finite with $|X| = \ell$ and $|Y| = m$ and that $Z$ is infinite. By projecting every vertex in the free product $X * Z$ in each copy of $X$ onto their respective base vertex, we get a quasi-isometry to the $\ell$-fold free product $Z * \cdots * Z$. If $\ell \neq 1$, this is quasi-isometric to $Z * Z$ by \cref{prop: X * Y is bilipschitz equivalent to X * (Y * Y)}. Similarly, if $m \neq 1$, then $Y * Z$ is also quasi-isometric to $Z * Z$.

    Finally, suppose that $X$ and $Y$ are infinite and that $Z$ is finite with $|Z| = n$. If $n = 1$, then $X * Z$ is quasi-isometric to $X$ and $Y * Z$ is quasi-isometric to $Y$. If $n \neq 1$, then as before $X * Z$ is quasi-isometric to $X * X$ and $Y * Z$ is quasi-isometric to $Y * Y$. Hence, they are quasi-isometric by assumption or by \cref{thm: quasi-isometric infinite factors implies quasi-isometric free products of graphs}.
\end{proof}

Similarly, \cref{prop: X * Y is bilipschitz equivalent to X * (Y * Y)} also extends to certain finite factors if we consider quasi-isometries instead of bilipschitz equivalences.

\begin{prop}\label{prop: X * Y is quasi-isometric to X * (Y * Y)}
    Let $X$ and $Y$ be homogeneous graphs. Then the free products $X * Y$ and $X * (Y * Y)$ are quasi-isometric, except in the following cases:
    \begin{itemize}
        \item $X$ or $Y$ has only a single vertex.

        \item $X$ and $Y$ have exactly two vertices.
    \end{itemize}
\end{prop}

\begin{proof}
    If all graphs are infinite, then the statement follows from \cref{prop: X * Y is bilipschitz equivalent to X * (Y * Y)}. Recall that under the assumptions on the number of vertices, the free products of graphs $X * (Y * Y)$ and $(X * Y) * Y$ are quasi-isometric. The other cases follow similarly to the proof of \cref{thm: quasi-isometric factors implies quasi-isometric free products of graphs}: If $X$ and $Y$ are both finite, then both $X * (Y * Y)$ and $X * Y$ are quasi-isometric to $T_3$. If $X$ is finite and $Y$ is infinite, then both $X * (Y * Y)$ and $X * Y$ are quasi-isometric to $Y * Y$. If $X$ is infinite and $Y$ is finite, then both $(X * Y) * Y$ and $X * Y$ are quasi-isometric to $X * X$.
\end{proof}

\section{Generalization to tdlc groups}\label{sec: generalization to tdlc groups}

In this section, we will prove \cref{thm: Papasoglu-Whyte for tdlc groups} and thus generalize the main theorem of Papasoglu--Whyte to tdlc groups. We begin with the implication that quasi-isometric vertex groups give rise to quasi-isometric fundamental groups of the corresponding graphs of tdlc groups. We first treat the cases of a single amalgam and a single HNN extension. The essence of the following arguments is to quasi-isometrically deform Cayley--Abels graphs so that \cref{thm: quasi-isometric factors implies quasi-isometric free products of graphs} or a suitable variant of \cref{prop: bilipschitz class of free products of graphs} becomes applicable. Recall that for a compactly generated tdlc group $G$ with a compact open subgroup $U$, we denote by $\Gamma_{G, U}$ a Cayley--Abels graph of $G$ with vertex set $G/U$ coming from a compact generating set $U \cup \{g_1, \ldots, g_n\}$, see \cref{constr: definition of a Cayley--Abels graph}.

\begin{prop}\label{prop: quasi-isometric factors implies quasi-isometric amalgams}
    Let $A_1, A_2, B_1$ and $B_2$ be compactly generated tdlc groups. For $i = 1, 2$ let $C_i$ be a common compact open proper subgroup of $A_i$ and $B_i$. If $A_1$ and $A_2$ are quasi-isometric and $B_1$ and $B_2$ are quasi-isometric, then the amalgams $A_1 *_{C_1} B_1$ and $A_2 *_{C_2} B_2$ are quasi-isometric, unless one $C_i$ has index $2$ in both $A_i$ and $B_i$. Furthermore, without any assumption on the indices, $A *_C B$ is quasi-isometric to the free product of graphs $\Gamma_{A, C} * \Gamma_{B, C}$.
\end{prop}

\begin{proof}
    By the normal form theorem for amalgams, the free product $\Gamma_{A_i,C_i} * \Gamma_{B_i,C_i}$ is a Cayley--Abels graph of $A_i *_{C_i} B_i$. Since Cayley--Abels graphs are homogeneous, under the assumption on the indices the statement follows from \cref{thm: quasi-isometric factors implies quasi-isometric free products of graphs}.
\end{proof}

For HNN extensions $A *_C$ more preparation is needed, since there are two associated compact open subgroups $C$ and $C'$ to take care of. In the discrete setting, $C$ and $C'$ are isomorphic finite subgroups. Papasoglu and Whyte use Hall's marriage theorem to obtain a common left transversal $N$ of $C$ and $C'$ in $A$ and in the tree of spaces for $A *_C$ project vertices in copies of $A$ onto $N$. However, if $C$ and $C'$ are infinite, the indices of $C \cap C'$ in $C$ and $C'$ can be different and a common left transversal need not exist. Nevertheless, the construction of a suitable net in $A$ remains a crucial step in the argument.

\begin{defn}
    Let $\Gamma$ be a graph and let $R, S > 0$. An $(R, S)$-\emph{net} in $\Gamma$ is a set $N$ of vertices of $\Gamma$ satisfying the following:
    \begin{itemize}
        \item For distinct vertices $m, n \in N$ we have $d(m, n) \ge R$.

        \item For every vertex $v \in \Gamma$ we have $d(v, N) \le S$.
    \end{itemize}

    We give $N$ a graph structure by joining vertices $m, n \in N$ by an edge whenever $d(m, n) \le 2S$.
\end{defn}

For any locally finite graph $\Gamma$ and any $R > 0$ there exist $(R, R)$-nets. Moreover, if $N$ is a net in $\Gamma$, then the inclusion $N \xhookrightarrow{} \Gamma$ is a quasi-isometry. The metric induced on $N$ as a subset of $\Gamma$ is bilipschitz equivalent to the graph metric on $N$.

\begin{prop}\label{prop: quasi-isometric factors implies quasi-isometric HNN extensions}
    Let $A$ and $B$ be compactly generated tdlc groups. Let $C$ and $C'$ be isomorphic compact open subgroups of $A$, and let $D$ and $D'$ be isomorphic compact open subgroups of $B$. If $A$ and $B$ are quasi-isometric, then the HNN extensions $G = A *_C$ and $H = B *_D$ are quasi-isometric, unless $A = C = C'$ or $B = D = D'$. 
    
    Assume $[A : C] \ge [A : C']$. Then, under no assumption on properness of the subgroups, $A *_C$ is quasi-isometric to the free product of graphs $\Gamma_{A, C} * T_2$.
\end{prop}

We prove this proposition in several steps. We first consider the case that $A$ and $B$ are non-compact. Since $C$ and $C'$ are compact open subgroups of $A$, the indices $m = [C : C \cap C']$ and $n = [C' : C \cap C']$ are finite. Let $A$ be compactly generated by $C$ and finitely many elements $g_1, \ldots, g_k$. Then $G$ is compactly generated by $C$, the elements $g_1, \ldots, g_k$ and the stable letter $t$. Denote by $X$ the Cayley--Abels graph $\Gamma_{A, C}$. The Cayley--Abels graph $\Gamma_{G, C}$ can be viewed as a tree of spaces consisting of copies of $X$. In contrast to the free product of graphs, each vertex in a copy of $X$ has $n$ outgoing edges (corresponding to multiplication by $t$) leading to the same copy of $X$, and $m$ incoming edges (corresponding to multiplication by $t^{-1}$) each coming from a different copy of $X$. This highlights some parallels between the tree of spaces $\Gamma_{G, C}$ and the free product $X * T_{m+n}$.

\tikzset{->-/.style={decoration={
  markings,
  mark=at position .5 with {\arrow{>}}},postaction={decorate}}}

\begin{figure}[ht]
    
    \centering
    
        \begin{tikzpicture}

        \draw (0,0) ellipse (0.5 and 1);

        \draw (2,2) ellipse (0.5 and 1);
        \draw (-1,3) ellipse (0.5 and 1);
        \draw (-3,2) ellipse (0.5 and 1);

        \draw[->-,red] (-3,2.5) to (0,0.5);
        \draw[->-,red] (-1,3.5) to (0,0.5);
        \draw[->-,red] (0,0.5) to [bend left] (2,2.5);
        \draw[->-,red] (0,0.5) to (2,2);
        \draw[->-,red] (0,0.5) to [bend right] (2,1.5);

        \fill[black] (0,0.5) circle (2pt);
        \fill[black] (2,2.5) circle (2pt);
        \fill[black] (2,2) circle (2pt);
        \fill[black] (2,1.5) circle (2pt);
        \fill[black] (-1,3.5) circle (2pt);
        \fill[black] (-3,2.5) circle (2pt);

        \draw (2,-2) ellipse (0.5 and 1);
        \draw (-1,-3) ellipse (0.5 and 1);
        \draw (-3,-2) ellipse (0.5 and 1);

        \draw[->-,red] (-3,-1.5) to (0,-0.5);
        \draw[->-,red] (-1,-2.5) to (0,-0.5);
        \draw[->-,red] (0,-0.5) to [bend right] (2,-2.5);
        \draw[->-,red] (0,-0.5) to (2,-2);
        \draw[->-,red] (0,-0.5) to [bend left] (2,-1.5);

        \fill[black] (0,-0.5) circle (2pt);
        \fill[black] (2,-2.5) circle (2pt);
        \fill[black] (2,-2) circle (2pt);
        \fill[black] (2,-1.5) circle (2pt);
        \fill[black] (-1,-2.5) circle (2pt);
        \fill[black] (-3,-1.5) circle (2pt);
        
        \end{tikzpicture}

    \caption{A section of the tree of spaces $\Gamma_{G,C}$ for $m = 2$ and $n = 3$. The directed edges connecting different subgraphs are drawn in red and correspond to multiplication by $t$. The red edges induce $5$-regular graphs.}
        
\end{figure}
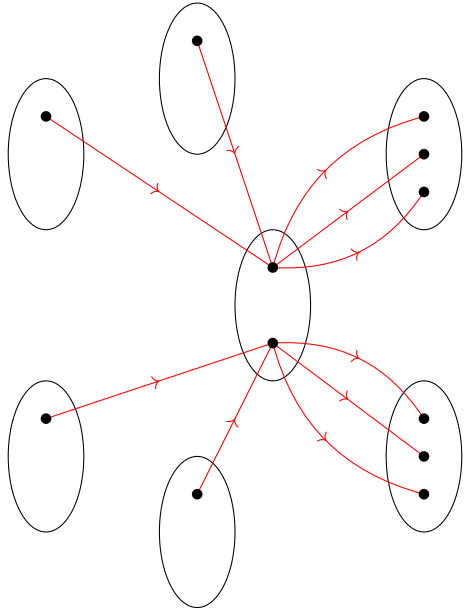

Let $Y = \Gamma_{B, D}$. We will construct nets $M \subseteq X$ and $N \subseteq Y$. Projecting in $\Gamma_{G, C}$ the vertices of $X$ onto $M$ yields a tree of spaces with most of the properties required in \cref{prop: bilipschitz class of free products of graphs}. However, the nets need not be homogeneous. Nevertheless, a careful construction of this new tree of spaces ensures that the proof strategy of \cref{prop: bilipschitz class of free products of graphs}, see \cite{papasoglu+whyte}*{Lemma 1.1}, still applies.

\begin{construction}\label{constr: quasi-isometry projection to net}
    Let $f \colon X \to Y$ be a quasi-isometry. Take an $(R, S)$-net $M$ in $X$ containing the base vertex $1_X$. Then its image $N = f(M)$ is a net in $Y$. If $R$ is large enough (depending only on the quasi-isometry constants of $f$), the restriction $f|_M \colon M \to N$ is injective, hence a bilipschitz equivalence. Since $Y$ is vertex-transitive, we may compose with an appropriate isometry so that $N$ contains the base vertex $1_Y$ and $f(1_X) = 1_Y$.

    Let $X_0$ be the base copy of $X$ in $\Gamma_{G,C}$ and let $M_0$ be the copy of $M$ in $X_0$. Consider a vertex $v \in V(X_0)$ with outgoing edges $(v, w_1), \ldots, (v, w_n)$ to a copy $X_1$ of $X$. Remove all these edges except $(v, w_1)$ and, using vertex-transitivity, identify $w_1$ with the base vertex of $X_1$.
    
    Let $(w, v_i)$ be the edges arriving at vertices $v_1, \ldots, v_n \in V(X_0)$ coming from a vertex $w$ from some adjacent copy $X_1$. If $S$ is large enough, there exists a vertex $m \in M_0$ such that $d(v_i, m) \le S$ for all $i$. Replace the edges $(w, v_i)$ by a single edge $(w, m)$, and identify $w$ with the base vertex of $X_1$.

    Project every vertex of $X_0$ to a nearest vertex of $M_0$; this vertex may not be unique. If $R \ge 3$, each vertex of $M_0$ has at least $3$ edges connecting it to different copies of $X$. The number of edges connecting a vertex of $M_0$ to a different copy of $X$ is bounded, depending only on $m, n, R$ and $S$.

    Let $X_1$ be a copy of $X$ adjacent to the base net $M_0$. Repeat the steps above for all vertices of $X_1$. Repeat this procedure first for every copy of $X$ adjacent to $M_0$ and then iterate for every copy of $X$ of distance $k$ to the base net $M_0$ for $k \to \infty$.
\end{construction}

\cref{constr: quasi-isometry projection to net} yields a quasi-isometry from the tree of spaces $\Gamma_{G, C}$ to a tree of spaces $\hat{X}$ built from copies of $M$ and trees of bounded valency, where the trees are built from the edges connecting distinct copies of $M$. Similarly, we obtain $\hat{Y}$ for $\Gamma_{H, D}$ using the net $N$. We want to construct a bilipschitz equivalence $\hat{X} \rightarrow \hat{Y}$. Before doing so, recall that trees of bounded valency are not only quasi-isometric but bilipschitz equivalent to each other.

\begin{prop}[\cite{papasoglu}]\label{prop: all bounded valency trees are bilipschitz equivalent}
    Let $K \ge 3$ and let $(T_i)_i$ be a family of trees such that for every $i$ any vertex $v \in T_i$ has valency between $3$ and $K$. Then there is an $L \ge 1$, depending only on $K$, such that between any two trees $T_i$ and $T_j$ there is an $L$-bilipschitz equivalence $T_i \to T_j$.
\end{prop}

\begin{construction}\label{constr: bilipschitz equivalence between nets}
    Following the proof idea of \cite{papasoglu+whyte}*{Lemma 1.1}, we will construct a bilipschitz equivalence $\varphi \colon \hat{X} \to \hat{Y}$. We start with the base net $M_0$ in $\hat{X}$. We have a bilipschitz equivalence $f \colon M_0 \to N_0$ mapping the base vertex of $M_0$ to the base vertex of $N_0$. Define $\varphi(m) := f(m)$ for every vertex $m \in M_0$.

    Let $v \in M_0$ be a vertex. Let $T$ be the tree obtained by restricting to edges connecting distinct copies of $M$ and taking the connected component containing $v$. Similarly, we get a tree $T'$ for $\varphi(v)$. By \cref{constr: quasi-isometry projection to net}, both trees have bounded valency and each vertex of those trees has valency at least $3$. Thus, we have a map 
    $$T \xrightarrow[]{\: h \:} T_3 \xrightarrow[]{\: \:} T_3 \xrightarrow[]{\: {h'}^{-1} \:} T'$$
    where $h \colon T \to T_3$ and $h' \colon T' \to T_3$ are the bilipschitz equivalences given by \cref{prop: all bounded valency trees are bilipschitz equivalent} and $T_3 \to T_3$ is an isometry of $T_3$ mapping $h(v)$ to $h'(\varphi(v))$. Since $h$ and $h'$ have bounded bilipschitz constants, the resulting bilipschitz equivalence $\hat{h} \colon T \to T'$ does too. It maps $v$ to $\varphi(v)$. Now define $\varphi(t) := \hat{h}(t)$.

    Now let $T$ be one of the trees connected to $M_0$ and let $t \in V(T) \setminus V(M_0)$. At $t$ we have a copy $M_t$ of $M$ and, by \cref{constr: quasi-isometry projection to net}, $t$ is the base vertex of $M_t$. Similarly, $\varphi(t)$ is the base vertex of a copy $N_t$ of $N$. Again, we have a bilipschitz equivalence $f_t \colon M_t \to N_t$ mapping $t$ to $\varphi(t)$. Define $\varphi(m) := f_t(m)$ for every $m \in M_t$. Repeat this procedure for every vertex $t \in V(T) \setminus V(M_0)$ and then for every tree $T$ connected to $M_0$.

    Finally, iterate this process for copies of $M$ and trees $T$ of increasing distance to the base net $M_0$. As the bilipschitz constant of every bilipschitz equivalence used is bounded depending only on the valency bound of the trees $T$ and the bilipschitz constant of $f \colon M \to N$, we get a bilipschitz equivalence $\varphi \colon \hat{X} \to \hat{Y}$.
\end{construction}

\begin{proof}[Proof of \cref{prop: quasi-isometric factors implies quasi-isometric HNN extensions}]
    By \cref{constr: quasi-isometry projection to net}, we have quasi-isometries $\Gamma_{G, C} \to \hat{X}$ and $\Gamma_{H, D} \to \hat{Y}$ where $\hat{X}$ and $\hat{Y}$ are trees of spaces. By \cref{constr: bilipschitz equivalence between nets}, there is a bilipschitz equivalence $\hat{X} \to \hat{Y}$. Composition of these maps yields a quasi-isometry $\Gamma_{G, C} \to \Gamma_{H, D}$. Thus, $G$ and $H$ are quasi-isometric.

    It remains to discuss the case of a compact vertex group. If $A \neq C$ or $A \neq C'$, then $A *_C$ is quasi-isometric to $T_3$ and has infinitely many ends. In contrast, if $A = C = C'$, then $X$ is quasi-isometric to $T_2$ and has exactly $2$ ends.

    Now let $A$ be a compactly generated tdlc group with a proper compact open subgroup $C$ with no assumption on compactness of $A$. Consider first the HNN extension $G = A *_{\id}$ corresponding to $\id \colon C \to C$. Then $\Gamma_{G, C}$ is quasi-isometric to $\Gamma_{A, C} * T_2$. Now let $C'$ be a compact open subgroup of $A$ and let $\varphi \colon C \to C'$ be an isomorphism. By the first part of \cref{prop: quasi-isometric factors implies quasi-isometric HNN extensions} the HNN extension $A *_{\varphi}$ is quasi-isometric to $A *_{\id}$. Hence, it is quasi-isometric to $\Gamma_{A, C} * T_2$.
\end{proof}

\Cref{prop: quasi-isometric factors implies quasi-isometric amalgams} and \cref{prop: quasi-isometric factors implies quasi-isometric HNN extensions} show one implication of \cref{thm: Papasoglu-Whyte for tdlc groups} for graphs of tdlc groups with a single edge. However, there is still one more base step to consider, namely a graph of tdlc groups without edges. Finally, by iterating these propositions, we extend our result to finite graphs of tdlc groups.

\begin{prop}\label{prop: X * X and X * T_2 are quasi-isometric}
    Let $X$ be an infinite homogeneous graph. Then $X * X$ and $X * T_2$ are quasi-isometric. In particular, if $A$ is a non-compact compactly generated tdlc group with isomorphic compact open subgroups $C$ and $C'$, then the HNN extension $A *_C$ and the amalgam $A *_C A$ are quasi-isometric. 
    
    Furthermore, if additionally $A$ has infinitely many ends, then $A$ is quasi-isometric to both the HNN extension $A *_C$ and the amalgam $A *_C A$.
\end{prop}

\begin{proof}
    Consider the Cartesian product of graphs $X \Box C_2$, which consists of two copies of $X$ with edges connecting them. By projecting$X \Box C_2$ onto one of the two copies of $X$, we get a quasi-isometry $X \Box C_2 \rightarrow X$. As $X$ is homogeneous, so is $X \Box C_2$. By \cref{thm: quasi-isometric infinite factors implies quasi-isometric free products of graphs}, the free products $X * X$ and $(X \Box C_2) * (X \Box C_2)$ are quasi-isometric. In $(X \Box C_2) * (X \Box C_2)$ project every vertex in each copy of $X \Box C_2$ onto one copy of $X$. This yields a quasi-isometry from $(X \Box C_2) * (X \Box C_2)$ to $X * T_2$.

    Now let $G$ be a compactly generated tdlc group with infinitely many ends and let $\Gamma$ be a Cayley--Abels graph of $G$. By Stallings' ends theorem, $G$ splits as an amalgam $A *_C B$ over a compact open subgroup such that $C$ has index at least $2$ in one factor and index at least $3$ in the other or as an HNN extension $A *_C$ over two isomorphic compact open subgroups $C$ and $C'$ such that $A \neq C$.
    
    Let $\Gamma_A = \Gamma_{A, C}$ and $\Gamma_B = \Gamma_{B, C}$ be the Cayley--Abels graphs. If $G = A *_C B$, then by \cref{prop: quasi-isometric factors implies quasi-isometric amalgams} and \cref{prop: X * Y is quasi-isometric to X * (Y * Y)} we get quasi-isometries
    $$\Gamma \rightarrow \Gamma_A * \Gamma_B \rightarrow \Gamma_A * \Gamma_A * \Gamma_B * \Gamma_B \rightarrow (\Gamma_A * \Gamma_B) * (\Gamma_A * \Gamma_B) \rightarrow \Gamma * \Gamma.$$
    If $G = A *_C$, then by \cref{prop: quasi-isometric factors implies quasi-isometric HNN extensions} and \cref{prop: X * Y is quasi-isometric to X * (Y * Y)} we get quasi-isometries
    $$\Gamma \rightarrow \Gamma_A * T_2 \rightarrow \Gamma_A * T_2 * T_2 \rightarrow \Gamma * T_2.$$
    In either case $\Gamma$ is quasi-isometric to $\Gamma * \Gamma$ or $\Gamma * T_2$ and thus to both.
\end{proof}

\begin{thm}\label{thm: quasi-isometric factors implies quasi-isometric graph of groups}
    Let $G$ and $H$ be compactly generated tdlc groups with infinitely many ends. Let $(\graph, \progr)$ and $(\mathcal{H}, \mathcal{B})$ be decompositions of $G$ and $H$ into finite graphs of tdlc groups with compact open edge groups. If $(\graph, \progr)$ and $(\mathcal{H}, \mathcal{B})$ have the same set of quasi-isometry classes of vertex groups, then $G$ and $H$ are quasi-isometric.
\end{thm}

\begin{proof}
    As $G$ decomposes as a finite graph of groups $(\graph, \progr)$, it can be written as iterated amalgams (for edges of a spanning tree $\tree$ of $\graph$) followed by iterated HNN extensions (for edges outside of $\tree$). Without loss of generality, assume that each occurring amalgam is non-trivial.
    
    First suppose that $G$ consists of a single vertex group with Cayley--Abels graph $\Gamma$. Depending on whether or not $\graph$ has edges, $G$ is quasi-isometric to $\Gamma * T_2$ either by \cref{prop: quasi-isometric factors implies quasi-isometric HNN extensions} and \cref{prop: X * Y is quasi-isometric to X * (Y * Y)} or by \cref{prop: X * X and X * T_2 are quasi-isometric} (as $G$ has infinitely many ends). Now suppose that $G$ has multiple vertex groups and denote their Cayley--Abels graphs by $\Gamma_1, \ldots, \Gamma_n$. As every amalgam is non-trivial, each $\Gamma_i$ has at least two vertices. \Cref{prop: quasi-isometric factors implies quasi-isometric amalgams} and \cref{prop: quasi-isometric factors implies quasi-isometric HNN extensions} show that $G$ is quasi-isometric to
    $$\Gamma_1 * \cdots * \Gamma_n * T_2 * \cdots * T_2$$
    where we get $T_2$ as a factor for every occurring HNN extension. If all factors are finite, i.e. $(\graph, \progr)$ has only compact vertex groups, then, under the assumption of having infinitely many ends, $G$ is quasi-isometric to $T_3$.

    So, suppose that not all vertex groups are compact. By \cref{prop: X * Y is quasi-isometric to X * (Y * Y)}, depending on whether or not $\graph$ is a tree, $G$ is quasi-isometric to
    $$\Gamma_1 * \cdots * \Gamma_n \text{ or } \Gamma_1 * \cdots * \Gamma_n * T_2.$$
    Next, we get rid of finite factors. Let $F$ be the free product of all finite factors. It is finite or quasi-isometric to $T_2$ or $T_3$. Let $\Lambda$ be the free product of all the infinite $\Gamma_i$. If $F$ is finite with $|F| \ge 2$, then we have quasi-isometries $F * \Lambda \to \Lambda * \Lambda \to \Lambda * T_2$ by \cref{prop: X * X and X * T_2 are quasi-isometric}. If $F$ is quasi-isometric to $T_2$ or $T_3$, then $F * \Lambda$ is again quasi-isometric to $\Lambda * T_2$. Finally, we may assume by \cref{prop: X * Y is quasi-isometric to X * (Y * Y)} that the factors of $\Lambda$ are pairwise not quasi-isometric. The only exception is if all factors of $\Lambda$ are quasi-isometric to a graph $\Delta$, in which case $\Lambda$ is quasi-isometric to $\Delta * \Delta$, which is quasi-isometric to $\Delta * T_2$ by \cref{prop: X * X and X * T_2 are quasi-isometric}.

    Let $S$ be the set of quasi-isometry classes of non-compact vertex groups of $(\graph, \progr)$ and let $\Delta_1, \ldots, \Delta_n$ be graphs representing each quasi-isometry class of $S$ once. Then, $G$ is quasi-isometric to\dots
    \begin{itemize}
        \item \dots $T_3$ if $S = \emptyset$, i.e. all vertex groups of $(\graph, \progr)$ are compact;

        \item \dots $\Delta_1 * T_2$ if $S$ contains only a single quasi-isometry class;

        \item \dots $\Delta_1 * \cdots * \Delta_n$ if $|S| \ge 2$.
    \end{itemize}
    The theorem now follows immediately from \cref{thm: quasi-isometric factors implies quasi-isometric free products of graphs}.
\end{proof}

Now we will prove the other implication of \cref{thm: Papasoglu-Whyte for tdlc groups}. Fortunately, the proofs of Papasoglu--Whyte hold almost verbatim with only little preparation needed.

\begin{prop}\label{prop: amalgam inclusion is quasi-isometric embedding}
    Let $A$ and $B$ be compactly generated tdlc groups with a common compact open subgroup $C$. Let $G = A *_C B$ be an amalgam of tdlc groups. Then the inclusions $A \xhookrightarrow{} G$ and $B \xhookrightarrow{} G$ are quasi-isometric embeddings.
\end{prop}

\begin{proof}
    As $A$ and $B$ are compactly generated, $G$ is compactly generated as well. Hence, there are finitely many $a_1, \ldots, a_m \in A$ and $b_1, \ldots, b_n \in B$ such that
    \begin{gather*}
        A = \left< C, a_1, \ldots, a_m \right>, B = \left< C, b_1, \ldots, b_n \right> \\ \text{ and } G = \left< C, a_1, \ldots, a_m, b_1, \ldots, b_n \right>.
    \end{gather*}
    We show that, using these compact generating sets, the inclusion $A \xhookrightarrow{} G$ is an isometric embedding. Let $a \in A$. Clearly $|a|_G \le |a|_A$. Let $g_1 \cdots g_k$ be a word over $C$ and the $a_i$ and $b_i$ evaluating to $a$. Write $g_1 \cdots g_k = \alpha_1 \beta_1 \cdots \alpha_\ell \beta_\ell$ where $\alpha_j$ is a word over $C$ and the $a_i$, and $\beta_j$ is a word over $C$ and the $b_i$. Consider reduced forms of elements of amalgams, see \cite{lyndon+schupp}*{Chapter IV.2}. As $a \in A$, one reduced form of $a$ is $a$ itself. Since all reduced forms have the same length, if $\alpha_1 \beta_1 \cdots \alpha_\ell \beta_\ell$ is a reduced form of $a$, then $g_1 \cdots g_k = \alpha_1$ is already a word over $C$ and the $a_i$. Now assume that $\alpha_1 \beta_1 \cdots \alpha_\ell \beta_\ell$ is not reduced. Then some $\alpha_i$ or $\beta_i$ evaluates to an element $c \in C$. In the case of $\alpha_i$ evaluating to $c$ the word
    $$
        \alpha_1 \beta_1 \cdots \alpha_{i-1} (\beta_{i-1} c \beta_i) \alpha_{i+1} \cdots \alpha_\ell \beta_\ell
    $$
    is of shorter or equal length evaluating to $a$. Iterating this process, we eventually get a reduced form of $a$ that is of shorter or equal length than $g_1 \cdots g_k$. Hence, $|a|_A \le |a|_G$ and $A \xhookrightarrow{} G$ is an isometric embedding. A symmetric argument shows that for the given compact generating sets $B \xhookrightarrow{} G$ is also an isometric embedding.
\end{proof}

For HNN extensions $G = A *_C$ of tdlc groups with associated subgroups $C$ and $C'$, we choose generating sets $A = \left< C, C', a_1, \ldots, a_n \right>$ and $G = \left< C, C', a_1, \ldots, a_n, t \right>$. Using reduced forms of elements of HNN extensions, the analogous statement of \cref{prop: amalgam inclusion is quasi-isometric embedding} also holds. We get the following.

\begin{prop}\label{prop: finite graph of groups inclusion is quasi-isometric embedding}
    Let $G$ be a compactly generated tdlc group with a decomposition $(\graph, \progr)$ into a finite graph of tdlc groups with compact open edge groups. Then for every vertex $v \in V(\graph)$ the inclusion $\progr_v \xhookrightarrow{} G$ is a quasi-isometric embedding.
\end{prop}

\begin{thm}\label{thm: quasi-isometric accessible groups implies quasi-isometric vertex groups}
    Let $G$ be an accessible tdlc group with a decomposition $(\graph, \progr)$ into a terminal graph of tdlc groups. If $H$ is an accessible tdlc group quasi-isometric to $G$, then any decomposition $(\mathcal{H}, \mathcal{B})$ of $H$ into a terminal graph of tdlc groups has the same set of quasi-isometry classes of one-ended vertex groups as $(\graph, \progr)$.
\end{thm}

\begin{proof}
    Let $v \in V(\mathcal{H})$ be a vertex such that $\mathcal{B}_v$ is one-ended. By \cref{prop: finite graph of groups inclusion is quasi-isometric embedding}, the inclusion $\mathcal{B}_v \xhookrightarrow{} H$ is a quasi-isometric embedding. As $G$ and $H$ are quasi-isometric, $\mathcal{B}_v$ embeds quasi-isometrically into $G$. Since the finitely many edge groups of $(\graph, \progr)$ are compact, they have bounded diameter. We follow the proofs of \cite{papasoglu+whyte}*{Lemma 3.2 and Theorem 3.1} verbatim and conclude that $\mathcal{B}_v$ is quasi-isometric to a one-ended vertex group of $(\graph, \progr)$. A symmetric argument shows that $(\graph, \progr)$ and $(\mathcal{H}, \mathcal{B})$ have the same set of quasi-isometry classes of one-ended vertex groups.
\end{proof}

\begin{proof}[Proof of \cref{thm: Papasoglu-Whyte for tdlc groups}]
    For tdlc groups with infinitely many ends, one implication follows immediately from \cref{thm: quasi-isometric factors implies quasi-isometric graph of groups}. For zero- and one-ended tdlc groups, there is nothing to show, as their only terminal decomposition into a graph of tdlc groups is the trivial one. For two-ended tdlc groups, it follows from the fact that these groups form a single quasi-isometry class and that they only split terminally as a single HNN extension $A *_C$ or a single amalgam $A *_C B$ with compact factors. The other implication follows immediately from \cref{thm: quasi-isometric accessible groups implies quasi-isometric vertex groups}.
\end{proof}

\section{Non-quasi-isometric simple tdlc groups}\label{sec: non-quasi-isometric simple tdlc groups}

We denote by $\S$ the class of compactly generated non-discrete tdlc groups that are topologically simple. Smith \cite{smith}*{Theorem 38} constructed $2^{\aleph_0}$ pairwise non-isomorphic abstractly simple groups in $\S$. He did this by introducing a group $\universal(M, N)$ associated to a pair of permutation groups $(M, N)$, which generalizes the Burger--Mozes \cite{burger+mozes}*{3.2} universal group $\universal(F)$ associated to a single permutation group $F$. In this section we will strengthen Smith's result by constructing $2^{\aleph_0}$ pairwise non-quasi-isometric abstractly simple groups in $\S$.

For a thorough introduction to and formal definition of the group $\universal(M, N)$, we refer the reader to Smith \cite{smith}. Let $X$ and $Y$ be two disjoint non-empty possibly infinite sets and let $M \le \Sym(X)$ and $N \le \Sym(Y)$ be permutation groups. Denote by $\tree$ the $(|X|, |Y|)$-biregular tree with bipartition $V(\tree) = V_X \cup V_Y$, where every vertex in $V_X$ has valency $|X|$ and every vertex in $V_Y$ has valency $|Y|$. For a vertex $v \in V(\tree)$ denote by $B(v)$ the set of vertices adjacent to $v$. We endow every permutation group with its permutation topology, i.e. the pointwise stabilizers of finite subsets form a basis of identity neighborhoods. This topology makes any permutation group Hausdorff and totally disconnected.

The \emph{Smith universal group} $G = \universal(M, N)$ is a subgroup of $\aut(\tree)$ and acts on $\tree$ such that for every vertex $v \in V(\tree)$ the induced subgroup of the stabilizer $G_v$ in $\Sym(B(v))$ is topologically isomorphic to $M$ if $v \in V_X$ and to $N$ if $v \in V_Y$. If $M$ and $N$ both are transitive permutation groups, then $G$ splits as the amalgam $G_v *_{G_{(v, w)}} G_w$ where $v \in V_X$ and $v \in V_Y$ are adjacent vertices. We summarize the most important topological properties of the Smith group.

\begin{prop}[\cite{smith}*{Theorem 1 and Theorem 30}]\label{prop: main properties of Smith groups}
    Let $M \le \Sym(X)$ and $N \le \Sym(Y)$ be non-trivial permutation groups. Then the following hold.

    \begin{enumerate}
        \item If $M$ and $N$ are closed, then $\universal(M, N)$ is closed.
        \item Suppose $M$ and $N$ are closed. Then $\universal(M, N)$ is locally compact if and only if all edge stabilizers of $\universal(M, N)$ are compact if and only if all point stabilizers of $M$ and of $N$ are compact.
        \item Suppose $M$ and $N$ are closed and compactly generated with compact point stabilizers. If $M$ and $N$ have finitely many orbits and one of them is transitive, then $\universal(M, N)$ is compactly generated.
        \item $\universal(M, N)$ is discrete if and only if $M$ and $N$ act freely.
        \item Suppose $M$ and $N$ are generated by point stabilizers. Then $\universal(M, N)$ is abstractly simple if and only if $M$ or $N$ is transitive.
    \end{enumerate}
\end{prop}

\begin{prop}\label{prop: Smith group point stabilizer quasi-isometric to M or N}
    Let $M \le \Sym(X)$ and $N \le \Sym(Y)$ be closed compactly generated permutation groups with compact point stabilizers. Then all point stabilizers in $G = \universal(M, N)$ are also compactly generated and:

    \begin{enumerate}
        \item point stabilizers $G_v$ for $v \in V_X$ are quasi-isometric to $M$;

        \item point stabilizers $G_w$ for $w \in V_Y$ are quasi-isometric to $N$.
    \end{enumerate}
\end{prop}

\begin{proof}
    By \cite{smith}*{Theorem 31}, all point stabilizers in $G$ are compactly generated. Now let $v \in V_X$ be a vertex. By \cite{smith}*{Proposition 18}, there is a topological isomorphism $G_v / G_{(B(v))} \to M$ where $G_{(B(v))}$ is the pointwise stabilizer of $B(v)$. As the intersection of compact edge stabilizers the pointwise stabilizer $G_{(B(v))}$ is compact as well. By \cite{cornulier+delaharpe}*{Proposition 4.C.12}, the projection $G_v \to M$ is a quasi-isometry. A symmetric argument shows (2).
\end{proof}

If $M$ and $N$ are transitive, the Smith group $\universal(M, N)$ splits as an amalgam. Thus, we can apply \cref{thm: Papasoglu-Whyte for tdlc groups}.

\begin{prop}\label{prop: non-quasi-isometric local groups imply non-quasi-isometric Smith groups}
    For $i = 1, 2$ let $M_i$ be a simple at most one-ended finitely generated group containing a proper non-trivial finite subgroup $Q_i$. Let $X_i = M_i / Q_i$ be the corresponding coset space. Then $M_i$ acts faithfully on $X_i$ such that $M_i \le \Sym(X_i)$. Consider $\Sym(3)$ as a subgroup of $\Sym(3)$. Then the Smith group $G_i = \universal(M_i, \Sym(3))$ is an abstractly simple group in $\S$. If $M_1$ and $M_2$ are not quasi-isometric, then neither are $G_1$ and $G_2$.
\end{prop}

\begin{proof}
    The permutation group $M_i$ has the conjugates of its finite subgroup $Q_i$ as finite point stabilizers. The intersection of all conjugates is a normal subgroup of $M_i$. As $Q_i$ is a proper subgroup and $M_i$ is simple, the intersection is trivial and $M_i$ acts faithfully on $X_i$. The subgroup generated by all conjugates is also normal in $M_i$. Since $Q_i$ is non-trivial and $M_i$ is simple, $M_i$ is generated by its point stabilizers. The permutation group $\Sym(3)$ acts non-freely and has finite point stabilizers which generate it. As both $M_i$ and $\Sym(3)$ are finitely generated and transitive, by \cref{prop: main properties of Smith groups} the Smith group $G_i$ is an abstractly simple group in $\S$ with compact edge stabilizers.
    
    Furthermore, the Smith group $G_i$ decomposes as the amalgam $G_{v_i} *_{G_{(v_i, w_i)}} G_{w_i}$ with $v_i \in V_{X_i}, w_i \in V_3$ adjacent, see \cite{smith}*{Remark after Lemma 22}. As $G_{v_i}$ is quasi-isometric to $M_i$ by \cref{prop: Smith group point stabilizer quasi-isometric to M or N}, it is also at most one-ended. Similarly, $G_{w_i}$ is quasi-isometric to $\Sym(3)$ and therefore compact. Thus, $G_{v_i} *_{G_{(v_i, w_i)}} G_{w_i}$ is a decomposition into a terminal graph of groups. By \cref{thm: Papasoglu-Whyte for tdlc groups}, $G_1$ and $G_2$ are not quasi-isometric.
\end{proof}

\begin{proof}[Proof of \cref{thm: Uncountably many quasi-isometry classes in S}]
    Let $(M_i)_i$ be the family of $2^{\aleph_0}$ pairwise non-quasi-isometric simple finitely generated torsion groups constructed in \cite{minasyan+osin+witzel}*{Corollary 1.4}. As they are torsion groups, they are at most one-ended by Stallings' ends theorem because non-trivial amalgams and HNN extensions contain elements of infinite order. Let $G_i = \universal(M_i, \Sym(3))$ be the Smith group as described in \cref{prop: non-quasi-isometric local groups imply non-quasi-isometric Smith groups}. Then $(G_i)_i$ is a family of $2^{\aleph_0}$ pairwise non-quasi-isometric abstracly simple groups in $\S$.
    
    By \cite{smith}*{Corollary 39}, there are precisely $2^{\aleph_0}$ isomorphism classes of groups in $\S$. Thus, there also are precisely $2^{\aleph_0}$ quasi-isometry classes of abstractly simple groups in $\S$.
\end{proof}

Lastly, we will slightly adjust the local actions and arrange that the Smith groups contain no lattices. Recall that a lattice in a locally compact group $G$ is a discrete subgroup $\Gamma$ such that $G/\Gamma$ admits a finite $G$-invariant measure. We do so by replacing $\Sym(3)$ by a abstractly simple non-discrete tdlc group not containing any lattices, e.g. a Neretin group.

\begin{ex}
    Let $(M_i)_i$ again be the family of $2^{\aleph_0}$ pairwise non-quasi-isometric simple finitely generated torsion groups constructed in \cite{minasyan+osin+witzel}*{Corollary 1.4}. Let $N = N_d$ be a Neretin group, i.e. the group of almost automorphisms of the $d$-regular tree $T_d$. It is a simple non-discrete tdlc group and it contains no lattices by \cite{bader+caprace+gelander+mozes}*{Theorem 1.1}. Let $Q \le N$ be a compact open subgroup such that $N$ acts on $Y = N/Q$. As $Q$ is simple, $N \le \Sym(Y)$ is a transitive permutation group with compact point stabilizers that generate $N$. Thus, $G_i = \universal(M_i, N)$ is a abstractly simple group in $\S$ by \cref{prop: main properties of Smith groups}.

    Suppose, $G_i$ contains a lattice $\Gamma$. For any open subgroup $O \le G_i$, the subgroup $\Gamma \cap O$ is a lattice in $O$, see \cite{caprace+monod}*{2.C}. In particular, the point stabilizer $(G_i)_w$ for $w \in V_Y$ contains a lattice. Recall that $(G_i)_w$ is quasi-isometric to the local action $N$ by \cref{prop: Smith group point stabilizer quasi-isometric to M or N}. More precisely, $N$ is a quotient of $(G_i)_w$ by a compact normal subgroup. By \cite{raghunathan}*{Theorem 1.13}, any lattice in $(G_i)_w$ induces a lattice in $N$. However, Neretin's group $N$ contains no lattice. Thus, neither does $G_i$. We obtain a family $(G_i)_i$ of $2^{\aleph_0}$ pairwise non-quasi-isometric abstracly simple groups in $\S$ which do not contain any lattice.
\end{ex}

  \bibliography{lit}
  \bibliographystyle{plain}

\end{document}